\documentclass[12pt]{amsart}

\newcommand{\pr}{\prod}
\newcommand{\A}{\ensuremath{\mathcal{A}}}

\newcommand{\C}{\ensuremath{\mathcal{C}}}
\newcommand{\we}{\wedge}\newcommand{\cd}{\cdots}
\newcommand{\Ess}{\tn{Ess}}\newcommand{\Es}{\tn{Ess}'}
\newcommand{\bq}{\begin{ques}}\newcommand{\eq}{\end{ques}}
\newcommand{\bc}{\begin{cor}}\newcommand{\ec}{\end{cor}}
\newcommand{\bcj}{\begin{conj}}\newcommand{\ecj}{\end{conj}}
\newcommand{\bl}{\begin{lem}}\newcommand{\el}{\end{lem}}

\newcommand{\vi}{\\[.1in]}\newcommand{\vii}{\\[.2in]}

\newcommand{\be}{\begin{enumerate}}\newcommand{\ee}{\end{enumerate}}
\newcommand{\bce}{\begin{center}}\newcommand{\ece}{\end{center}}
\newcommand{\bri}{\begin{flushright}}\newcommand{\eri}{\end{flushright}}
\newcommand{\bb}{\begin{block}}\newcommand{\eb}{\end{block}}
\newcommand{\bt}{\begin{thm}}\newcommand{\et}{\end{thm}}
\newcommand{\bpf}{\begin{proof}}\newcommand{\epf}{\end{proof}}
\newcommand{\bex}{\begin{ex}}\newcommand{\eex}{\end{ex}}
\newcommand{\bexr}{\begin{exr}}\newcommand{\eexr}{\end{exr}}
\newcommand{\bft}{\begin{fact}}\newcommand{\eft}{\end{fact}}
\newcommand{\brk}{\begin{rmk}}\newcommand{\erk}{\end{rmk}}
\newcommand{\ba}{\begin{align*}}\newcommand{\ea}{\end{align*}}
\newcommand{\tn}{\textnormal}

\newcommand{\bit}{\begin{itemize}}\newcommand{\eit}{\end{itemize}}
\newcommand{\os}{\overset}\newcommand{\us}{\underset}

\newcommand{\bs}[3]{\us{#1}{\os{#3}{#2}}} 
\newcommand{\bcm}{}
\newcommand{\ol}{\overline}\newcommand{\ul}{\underline}
\newcommand{\hf}{\hfill}
\newcommand{\cci}{\Circle}
\newcommand{\fr}{\frac}

\newcommand{\rr}{\ensuremath{\mathbf{R}}}

\newcommand{\bob}{\begin{ob}}\newcommand{\eob}{\end{ob}}
\newcommand{\bd}{\begin{defn}}\newcommand{\ed}{\end{defn}}
\newcommand{\bp}{\begin{prop}}\newcommand{\ep}{\end{prop}}

\newcommand{\eh}{\emph}\newcommand{\al}{\alpha}
\newcommand{\sub}{\subseteq}

\renewcommand{\i}{\item}
\newcommand{\mb}{\mbox}
\newcommand{\te}{\text}\newcommand{\ph}{\phantom}
\newcommand{\wt}{\widetilde}
\newcommand{\lef}{\left}\newcommand{\ri}{\right}

\newcommand{\then}{\Longrightarrow}

\newcommand{\di}{\displaystyle}\renewcommand{\a}{\bm{a}}
\renewcommand{\b}{\beta}

\newcommand{\x}{\ensuremath{\bm{x}}}

\newcommand{\np}{\newpage}

\usepackage[dvipdfmx]{graphicx}
\usepackage[dvipsnames]{xcolor}
\usepackage{float,afterpage}
\usepackage{asymptote,layout}
\usepackage{wrapfig,epic}

\usepackage{geometry}
\usepackage{exscale,latexsym,bm}
\usepackage{amssymb,enumerate,amsmath,amsthm,amsfonts}
\usepackage{verbatim,fancybox,wasysym,fancyhdr,type1cm}
\usepackage[frame,all,poly,curve,knot,arrow]{xy}
\usepackage{colortbl}
\usepackage{boxedminipage}
\usepackage{multirow}

\graphicspath{{./figs/}}
\theoremstyle{plain}
\newtheorem{thm}{Theorem}[section]
\newtheorem{lem}[thm]{Lemma}
\newtheorem{prop}[thm]{Proposition}\newtheorem{cor}[thm]{Corollary}
\newtheorem{conj}[thm]{Conjecture}
\newtheorem{ques}[thm]{Question}
\newtheorem{keylem}[thm]{Key Lemma}

\theoremstyle{definition}
\newtheorem{ex}[thm]{Example}\newtheorem{exr}{Exercise}
\newtheorem{defn}[thm]{Definition}\newtheorem{rmk}[thm]{Remark}
\newtheorem{fact}[thm]{Fact}
\newtheorem{block}[thm]{}\newtheorem{ob}[thm]{Observation}

\begin{document}
\title[A directed graph structure of ASMs]{A directed graph structure of alternating sign matrices}
\author{Masato Kobayashi}
\date{\today}
\address{Department of Engineering\\
Kanagawa University,
3-27-1 Rokkaku-bashi, Yokohama 221-8686, Japan.}
\email{kobayashi@math.titech.ac.jp}
\thanks{This was already published as 
M. Kobayashi, A directed graph structure of alternating sign matrices,
Linear Algebra and its Applications \textbf{519} (2017), 164-190.}

\maketitle
\begin{abstract} 
We introduce a new directed graph structure into the set of alternating sign matrices. This includes Bruhat graph (Bruhat order) of the symmetric groups as a subgraph (subposet).\\
\indent Drake-Gerrish-Skandera (2004, 2006) gave characterizations of Bruhat order in terms of total nonnegativity (TNN) and subtraction-free Laurent (SFL) expressions for permutation monomials. With our directed graph, we extend their idea in two ways: first, from permutations to alternating sign matrices; second, $q$-analogs (which we name $q$TNN and $q$SFL properties). 
 As a by-product, we obtain a new kind of permutation statistic, the signed bigrassmannian statistics, using Dodgson's condensation on determinants.\\
\begin{center}
\end{center}
\end{abstract}
\np

\tableofcontents
\section{Introduction}
\label{sec1}
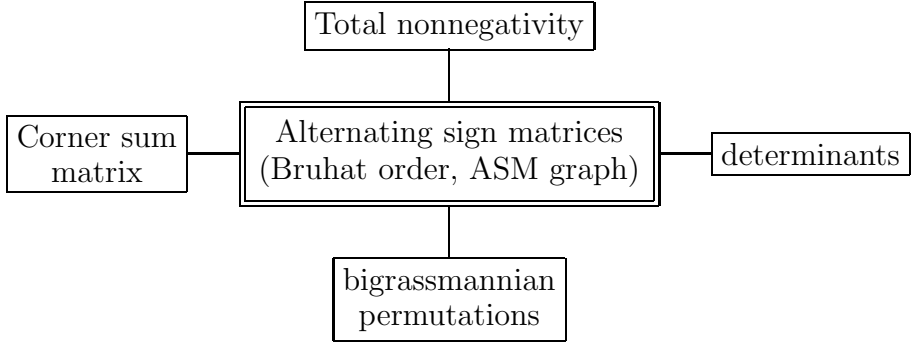
\begin{figure}
\label{related}
\caption{ASMs and Related ideas}
\bce{
\[\xymatrix@=7mm{
&*+[F]\txt{Total nonnegativity}&\\
*+[F]\txt{Corner sum\\ matrix}\ar@{-}[r]&*++[F=]\txt{Alternating sign matrices\\(Bruhat order, ASM graph)}\ar@{-}[d]\ar@{-}[u]\ar@{-}[r]\ar@{-}[l]
&*+[F]\txt{determinants}\\
&*+[F]\txt{bigrassmannian\\ permutations}&
}\]
}\ece
\end{figure}





\subsection{Bruhat order}
\emph{Bruhat order} has been of great importance in the combinatorial matrix theory; there are many equivalent characterizations of this order. For example, one is the transitive closure of the binary relation $u\to v$ on $S_n$ to mean $v=ut$ for some transposition $t$ and $\ell(u)<\ell(v)$ (with $\ell$ the number of inversions). 
Other variations are: 
\bit{
\i Entrywise order on Corner sum matrices; for example, see Brualdi-Deaett \cite{bd} and Fortin \cite{fortin}.
\i Lascoux-Sch\"{u}tzenberger's monotone triangles \cite{ls}.
}\eit

In addition to this list, Drake-Gerrish-Skandera \cite{dgs1,dgs2} found several new characterizations of Bruhat order in terms of permutation monomials:
\bft{Let $u, v\in S_n$. Then the following are equivalent:
\be{\item $u\le v$ in Bruhat order.
\i the polynomial $x_{1u(1)}\cdots x_{nu(n)}-x_{1v(1)}\cdots x_{nv(n)}$ is TNN.
\i the polynomial $x_{1u(1)}\cdots x_{nu(n)}-x_{1v(1)}\cdots x_{nv(n)}$ has (SFL) property.
}\ee
}\eft

Here TNN and SFL abbreviate ``Totally NonNegative" and ``Subtraction-Free Laurent expression", respectively; we give details of these terms later.

\subsection{Main results}
The aim of this article is simply to generalize Drake-Gerrish-Skandera's result above in two ways (Theorem \ref{qth}); first, permutations to \eh{alternating sign matrices} (ASMs); second, we will establish a  $q$-analog of their result. We also observe some byproducts on permutation statistics (Theorems \ref{bff} and \ref{mth3}). 
For this purpose, we introduce a new directed graph structure to ASMs as in the title; we call it \eh{ASM graph} (Figure \ref{related}). 

\subsection{Outline}
This articles consists of six sections. 
Section \ref{sec2} serves preliminaries on permutations and alternating sign matrices. Section \ref{sec3} gives a precise definition of ASM graph with notions of essential rectangles and bigrassmannian statistics; in particular, Key Lemma \ref{answer} will play a role in the sequel. In Section \ref{sec4}, we review Total nonnegativity and Subtraction-Free Laurent property. In Section \ref{sec5}, we give proofs of main results. We end with the conclusion remark in Section \ref{sec6}.

To better understand the global picture of our discussion, it is helpful 
to keep Figure \ref{mind} in mind.

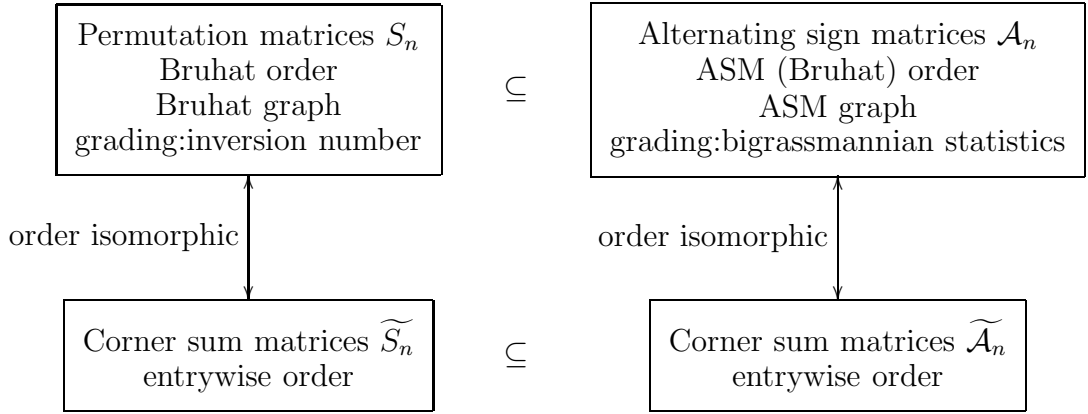
\begin{figure}
\caption{Global picture of our discussion}\label{mind}
\mb{}\vii
\begin{center}
\xymatrix@=7mm{
*++[F]\txt{Permutation matrices $S_{n}$\\Bruhat order\\
Bruhat graph\\
grading:inversion number
}
\ar@{<->}[dd]_-*\txt{order isomorphic}
&*\txt{$\sub$}&
*++[F]\txt{Alternating sign matrices $\mathcal{A}_{n}$\\ASM (Bruhat) order\\
ASM graph\\
grading:bigrassmannian statistics
}\ar@{<->}[dd]_-*\txt{order isomorphic}\\
&&&
\\
*++[F]\txt{Corner sum matrices $\wt{S_{n}}$\\entrywise order
}
&\sub&
*++[F]\txt{Corner sum matrices $\wt{\mathcal{A}_{n}}$\\entrywise order
}
}
\end{center}
\end{figure}
%
%

\subsection{Additional note}
At the time of writing this article, 
the author found that 
there are overlap with 
the recent article 
\begin{quote}
R. Brualdi, M. Schroeder, Alternating sign matrices and 
their Bruhat order, to appear in Discrete Math.
\end{quote}
Brualdi and Schroeder discuss the sequential construction of an ASM from the unit matrix (corresponding to our directed graph structure) as well as an enumerative property of B-rank function for ASMs (corresponding to bigrassmannian statistics in our terminology).

\begin{center}
{\bf acknowledgment.}\\
The author would like to thank the editor as well as the anonymous referee for helpful comments for improvement of the manuscript.
\end{center}

\section{Alternating sign matrices}\label{sec2}

\begin{figure}
\caption{$(\mathcal{A}_3, \le)$}
\label{a3}
\[
\xymatrix@=5mm{
&*+{\left[\begin{array}{ccc}0  &0   &1   \\ 0 & 1  & 0  \\1  & 0  &0  \end{array}\right]}
\ar@{-}[dl]\ar@{-}[dr]
&\\
*+{\left[\begin{array}{ccc}0 &1   &0   \\0  &0   &1   \\ 1 & 0  &0  \end{array}\right]}
\ar@{-}[dr]
&&*+{\left[\begin{array}{ccc}0 &0   &1   \\ 1 &0  &0   \\0  & 1  &0  \end{array}\right]}\ar@{-}[dl]
\\
&*+{\left[\begin{array}{ccc} 0 &1   &0   \\1  &-1   &1   \\0 &1   &0  \end{array}\right]}
\ar@{-}[dl]\ar@{-}[dr]
\\
*+{\left[\begin{array}{ccc} 1 & 0  &0   \\ 0 &0  &1   \\0  &1   &0  \end{array}\right]}
\ar@{-}[dr]
&&*+{\left[\begin{array}{ccc}  0& 1  &0  \\ 1 & 0  &0   \\0  &0   &1  \end{array}\right]}
\ar@{-}[dl]
\\
&*+{\left[\begin{array}{ccc}  1& 0  &0   \\0  &1   &0   \\ 0 &0   &1  \end{array}\right]}
&
}\]
\end{figure}
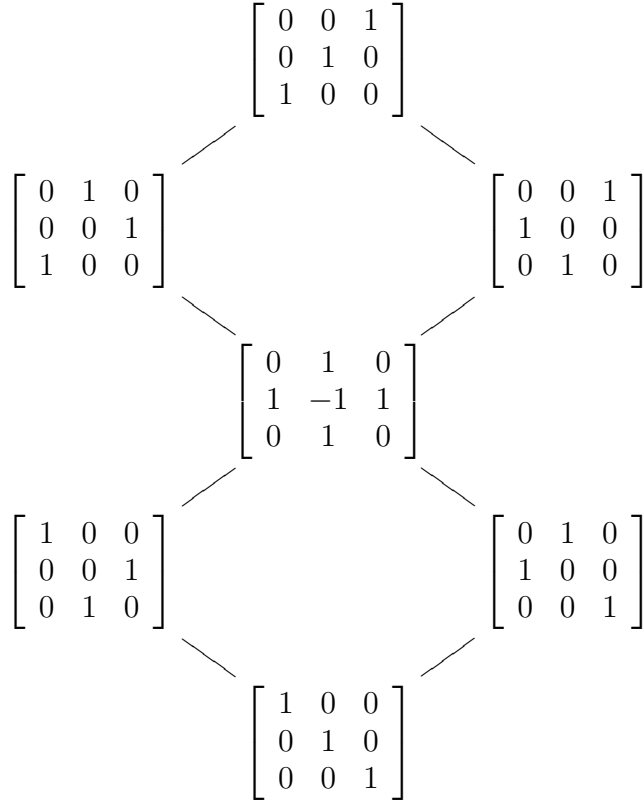

\begin{figure}
\caption{$(\wt{\mathcal{A}}_3, \le)$}
\label{q3}
\[
\xymatrix@=5mm{
&*+{\left[\begin{array}{ccc}0  &0   &1   \\ 0 & 1  & 2  \\1  & 2  &3  \end{array}\right]}
\ar@{-}[dl]\ar@{-}[dr]
&\\
*+{\left[\begin{array}{ccc}0 &1   &1   \\0  &1   &2   \\ 1 & 2  &3  \end{array}\right]}
\ar@{-}[dr]
&&*+{\left[\begin{array}{ccc}0 &0   &1   \\ 1 &1  &2   \\1  & 2  &3  \end{array}\right]}\ar@{-}[dl]
\\
&*+{\left[\begin{array}{ccc} 0 &1   &1   \\1  &1   &2   \\1 &2   &3  \end{array}\right]}
\ar@{-}[dl]\ar@{-}[dr]
\\
*+{\left[\begin{array}{ccc} 1 & 1  &1   \\ 1 &1  &2   \\1  &2   &3  \end{array}\right]}
\ar@{-}[dr]
&&*+{\left[\begin{array}{ccc}  0& 1  &1   \\ 1 & 2  &2   \\1  &2   &3  \end{array}\right]}
\ar@{-}[dl]
\\
&*+{\left[\begin{array}{ccc}  1& 1  &1   \\1  &2   &2   \\ 1 &2   &3  \end{array}\right]}
&
}\]
\end{figure}


For a positive integer $n$, let $[n]$ denote the set $\{1, 2, \dots, n\}$. 
Throughout this article, we assume that $n\ge 3$ to avoid some triviality. By $S_n$ we mean the symmetric group on $[n]$. To represent permutations, we often use \eh{one-line notation}: ``$u=i_1\cdots i_n$" with $i_k\in [n]$ means $u(k)=i_k$. For instance, $u=231$ means $u(1)=2, u(2)=3$ and $u(3)=1$.
Below, $A=(a_{ij})$ and $B=(b_{ij})$ are square matrices of size $n$ unless otherwise specified. For convenience, we write $a_{ij}$ as well as $A(i, j)$ for a matrix entry of $A$.\\

\subsection{Alternating sign matrices}\label{sec22}
We begin with definitions of permutation matrices and alternating sign matrices. 

\bd{We say that $A$ is a \emph{permutation matrix} (PM) if 
there exists a unique permutation $u\in S_n$ such that $a_{ij}=1$ if $j=u(i)$ and $a_{ij}=0$ otherwise.
}\ed

In this way, we often identify a permutation and a permutation matrix.

\bd{We say that $A$ is an \emph{alternating sign matrix} (ASM) if
for all $(i, j)\in [n]^2$, 
we have 
\[\begin{array}{lll}\di a_{ij}\in \{-1, 0, 1\},  &   &\di\sum_{k=1}^j a_{ik}\in \{0, 1\},   \vi \di\sum_{k=1}^i a_{kj}\in \{0, 1\} &\te{ and }   &  \di\sum_{k=1}^n a_{ik}= \sum_{k=1}^n a_{kj}=1.\end{array}\]

Denote by $\A_n$ the set of all alternating sign matrices of size $n$.
}\ed
Note that every PM is an ASM. Say an ASM is \emph{proper} if it is not a PM; in other words, an ASM is proper if and only if it has a $-1$ entry. 
Figure \ref{a3} shows seven ASMs in $\A_3$; the only one matrix 
in the middle is proper.


\subsection{Corner sum matrices}

\bd{The \emph{corner sum matrix} of $A\in \A_n$ is the $n$ by $n$ matrix $\wt{A}$ defined by 
\[
\wt{A}(i, j)=\sum_{p\le i, q\le j}a_{pq}\]
for all $i, j$. Denote by $\wt{\A}_n$ the set of all such matrices. 
}\ed

\begin{ex}\hf\vi
For $A=\left[\begin{array}{ccc} 0 &1   &0   \\1  &-1   &1   \\0 &1   &0  \end{array}\right]$, we have 
$\wt{A}=\left[\begin{array}{ccc} 1 & 1  &1   \\ 1 &1  &2   \\1  &2   &3  \end{array}\right].$
\end{ex}

\brk{\hfill
\be{\i Entries of each corner sum matrix are weakly increasing along rows and columns:
$\wt{A}(i, j)\le \wt{A}(k, l)$ if $i\le k$ and $j\le l$.
\i It is convenient to define $a_{ij}=0$ and $\wt{A}(i, j)=0$ whenever $i$ or $j$ is $0$. Then, we can recover each entry $a_{ij}$ from entries of $\wt{A}$: 
\[a_{ij}=\wt{A}(i, j)+\wt{A}(i-1, j-1)-\wt{A}(i, j-1)-\wt{A}(i-1, j) \mb{ for } i, j\ge 1.\]
The correspondence $A\leftrightarrow \wt{A}$ between $\A_n$ and $\wt{\A}_n$ is in fact a bijection; see Figures \ref{a3} and \ref{q3}, for example. 
}\ee
}\erk

The following criterion will be useful later.
\bft[Robbins-Rumsey {\cite[p.172, Lemma 1]{mmr2}}]{\label{cri} Let $X$ be a square matrix of size $n$.
Then $X\in \wt{\A}_n$ if and only if
$X(i, n)=X(n, i)=i$ for all $i$ and 
$X(i, j)-X(i-1,j)\in\{0, 1\}$, $X(i, j)-X(i, j-1)\in\{0, 1\}$ for all $i, j$.
}\eft

\section{Bruhat graph and ASM graph}\label{sec3}

In this section, we give a precise definition of \eh{ASM graph}; this is a directed graph structure of ASMs as in the title of this article.
We first review the definition of Bruhat graph on permutations; we will see that it is a certain subgraph of ASM graph.
\subsection{Bruhat graph}

For natural numbers $i<j\le n$, let $t_{ij}$ denote the transposition interchanging $i$ and $j$.
Say a pair $(i, j)$ is an \emph{inversion} of a permutation $u\in S_n$ if $i<j$ and $u(i)>u(j)$.
Let $\ell(u)$ be the number of inversions of $u$. Write $u\to v$ if $v=ut_{ij}$ and $\ell(u)<\ell(v)$ (equivalently, $(i, j)$ is an inversion of $v$). 
The directed graph $(S_n, \to) $ is the \emph{Bruhat graph}.

\bex{We have the edge relation $1342\to 4312 $; in terms of permutation matrices, we understand this relation as 
\[
\left(\begin{array}{cccc}1 & 0 & 0 & 0 \\0 & 0 & 1 & 0 \\0 & 0 & 0 & 1 \\0 & 1 & 0 & 0\end{array}\right)\to \left(\begin{array}{cccc}0 & 0 & 0 & 1 \\0 & 0 & 1 & 0 \\1 & 0 & 0 & 0 \\0 & 1 & 0 & 0\end{array}\right)\]
interchanging first and fourth columns (first and third rows).
}\eex

\bd{Define \emph{Bruhat order} $u\le v$ in $S_n$ if there exists a directed path from $u$ to $v$.}\ed

This is indeed a partial order on $S_n$. Here are more details:

\bft[Chain Property]{$(S_n, \le)$ is a graded poset ranked by $\ell$.
In other words, if $u\le v$, then there exists a directed path $u=u_0\to u_1\to u_2\to \cdots \to u_k=v$ such that $\ell(u_i)-\ell(u_{i-1})=1$.
}\eft

We wish to extend Bruhat order to ASMs (recall that every PM is an ASM). However, we have to take care of the following two points:

\bit{
\i Transposing columns or rows of an ASM does not necessarily produce an ASM. Thus, we need to modify a definition of the edge relation.
\i Find a rank function on ASMs, instead of the inversion number, such that it is monotonically increasing along those directed edges.
}\eit

We solve these problems with a new definition of a directed edge relation  using \eh{corner sum matrices} and \eh{bigrassmannian statistics}.





\subsection{ASM order}

Make sure that there is an equivalent characterization of Bruhat order in terms of corner sum matrices (rather than entries of PMs):

\bft{The following are equivalent:
\be{\i $u\le v$ in Bruhat order in $S_n$.
\i $\wt{u}(i, j)\ge \wt{v}(i, j)$ for all $i, j\in[n]$.
}\ee
}\eft

This idea naturally extends to ASMs:

\bd{Define \emph{ASM order} $A\le B$ in $\A_n$ if $\wt{A}(i, j)\ge \wt{B}(i, j)$ for all $i, j\in[n]$.
}\ed
By abuse of language, we also call this ``Bruhat order".
Hence $(\A_{n}, \le)$ is now a poset.


\brk{
Indeed, $(\A_n, \le)$ is a finite distributive lattice 
as the MacNeille completion of Bruhat order (the smallest lattice which contains $(S_{n}, \le)$ as a subposet). See Reading \cite{reading} for some more details.
}\erk


\subsection{Essential rectangles}
As before, let $A$ be an ASM. Consider integers $i, j, k, l\in[n]$ such that $i< j$ and $k< l$.
Let 
\[
R_{ij}^{kl}=\{(p, q)\in[n]^2\mid i\le p< j \tn{\ph{a}and } k\le q< l\}\]
be rectangular positions in a matrix (here, $i\le p$ and $k\le q$ are weak inequalities while $p< j$ and $q<l$ are strict).
\bd{We say that $R_{in}^{kl}$ is an \eh{essential rectangle} for $A$
if 
\[\begin{array}{lcl} \wt{A}(p, k)=\wt{A}(p, k-1), &   & \wt{A}(p, l)=\wt{A}(p, l-1)+1,
  \\\wt{A}(i, q)=\wt{A}(i-1, q),
    & \te{and}  & \wt{A}(j, q)=\wt{A}(j-1, q)+1 \end{array}\]
for all $(p, q)\in R_{ij}^{kl}$.
Similarly, say $R_{ij}^{kl}$ is a \eh{dual essential rectangle} for $A$
if 
\[\begin{array}{lcl} \wt{A}(p, k)=\wt{A}(p, k-1)+1, &   & \wt{A}(p, l)=\wt{A}(p, l-1),
  \\\wt{A}(i, q)=\wt{A}(i-1, q)+1,
    & \te{and}  & \wt{A}(j, q)=\wt{A}(j-1, q) \end{array}\]
for all $(p, q)\in R_{ij}^{kl}$.
We call such conditions (dual) \eh{essential conditions}.
Denote by $E(A)$ ($E^*(A)$) the set of such (dual) rectangles for $A$.
}\ed

Recall that adjacent entries of any corner sum matrix differs only by 0 or 1. These conditions above describe ``boundary conditions" on these rectangular positions.
Note: we understand $\wt{A}(p, q)=0$ if $p$ or $q$ is $0$; we often omit these zero entries when we write a corner sum matrix.

\bex{
On the one hand, the permutation 4312 has an essential rectangle $R_{13}^{14}$ since 
\[
\wt{4312}=\left(\begin{array}{cccc}\ul{\,0\,} & \ul{\,0\,} & \ul{\,0\,} & 1 \\\ul{\,0\,} & \ul{\,0\,} & \ul{\,1\,} & 2 \\1 & 1 & 2 & 3 \\1 & 2 & 3 & 4\end{array}\right)
.\] 
On the other hand, the permutation 1342 has a dual essential rectangle $R_{13}^{14}$ since 
\[
\wt{1342}=\left(\begin{array}{cccc}\ul{\,1\,} & \ul{\,1\,} & \ul{\,1\,} & 1 \\\ul{\,1\,} & \ul{\,1\,} & \ul{\,2\,} & 2 \\1 & 1 & 2 & 3 \\1 & 2 & 3 & 4\end{array}\right)
.\] As we see, underlined positions indicate such rectangles.
}\eex

\bp{
Let $u\in S_n$ and $i<j$. Then the following are equivalent:
\be{\i $(i, j)$ is an inversion of $u$.
\i $R_{i, j}^{u(j), u(i)}$ is an essential rectangle for $u$.
}\ee
}\ep

\bpf{If $(i, j)$ is an inversion of $u$, then there exist two $1$s at $(i, u(i))$ and $(j, u(j))$ positions in the permutation matrix $u$. It follows from the definition of a corner sum matrix that $R_{ij}^{u(j)u(i)}$ satisfies the essential conditions described above. Conversely, if $R_{ij}^{u(j)u(i)}$ is an essential rectangle for $u$, then it is necessarily that $u(j)<u(i)$.
}\epf

\bd{For $i<j$ and $k< l$, let $\wt{R}_{ij}^{kl}$
be the $n$ by $n$ matrix such that its $(p, q)$-entry is 1 if $(p, q)\in R_{ij}^{kl}$ or $0$ otherwise. Define a \eh{rectangular operator} $\wt{r}_{ij}^{kl}:\wt{\A}_n\to \wt{\A}_n$
\[\wt{r}_{ij}^{kl}(\wt{A})=\begin{cases}
\wt{A}+\wt{R}_{ij}^{kl}&\te{if } R_{ij}^{kl}\in E(A),\\
\wt{A}-\wt{R}_{ij}^{kl}&\te{if } R_{ij}^{kl}\in E^*(A),\\
A &\te{otherwise.}
\end{cases}\]
}\ed

So this operator changes entries of a consecutive submatrix of entries of a corner sum matrix.

\begin{ex}
\[
\wt{1342}=
\wt{r}_{13}^{14}(\wt{4312})=
\underbrace{\left(\begin{array}{cccc}0 & 0 & 0 & 1 \\0 & 0 & 1 & 2 \\1 & 1 & 2 & 3 \\1 & 2 & 3 & 4\end{array}\right)}_{\wt{4312}}
+
\underbrace{\left(\begin{array}{cccc}1 & 1 & 1 & 0 \\1 & 1 & 1 & 0 \\0 & 0 & 0 & 0 \\0 & 0 & 0 & 0\end{array}\right)}_{\wt{R}_{13}^{14}}
.\] 
\end{ex}

Similarly, define an operator $r_{ij}^{kl}:\A_n\to \A_n$ 
with $r_{ij}^{kl}(A)$ being the ASM whose corner sum matrix is $\wt{r}_{ij}^{kl}(\wt{A})$.

\brk{\hfill
\be{\i Let us be careful: whenever $R_{ij}^{kl}\in E(A)$, is the resulting matrix $\wt{A}+\wt{R}_{ij}^{kl}$ an element of $\wt{\A}_n$? Yes. Indeed, adjacent entries of $\wt{A}+\wt{R}_{ij}^{kl}$ differ only by $0$ or $1$ (sharing the $n$-th row and column entries of $\wt{A}$). Fact \ref{cri} guarantees that $\wt{A}+\wt{R}_{ij}^{kl}$ is a corner sum matrix for some (unique) ASM.
\i Observe that $r_{ij}^{kl}$ is an involution, i.e., $(r_{ij}^{kl})^2A=A$.
}\ee
}\erk



With this idea, it is natural to introduce the following statistic for ASMs as (the negative of) a sum of entries of corner sum matrices.

\bd{For $i, j$, let $i\we j=\min\{i, j\}$.
For $A\in \A_n$, define the \emph{bigrassmannian statistic}
\[
\b(A)=
\sum_{i, j=1}^n(i\we j) -\sum_{i, j=1}^n\wt{A}(i, j).\]
}\ed

Here the constant $\sum i\we j$ comes for normalization so that $\b(e)=0$ where $e$ is the unit of $S_n$ so that $\wt{e}(i, j)=i\we j$.

Observe the following dichotomy: 
for each $R_{ij}^{kl}\in E(A)\cup E^*(A)$, we have 
either $\beta(r_{ij}^{kl}A)<\beta(A) \iff R_{ij}^{kl}\in E(A)$ or $\beta(r_{ij}^{kl}A)>\beta(A) \iff R_{ij}^{kl}\in E^*(A)$.
With notions of essential rectangles and this statistic, we are now ready to introduce ASM graph as a generalization of Bruhat graph.

\bd{Define an edge relation $A\bs{ij}{\to}{kl}B$ in $\A_n$ if $B=r_{ij}^{kl}(A)$ and $\b(A)<\b(B)$. 
By $A\to B$ we mean $A\bs{ij}{\to}{kl}B$ for some $i, j, k, l$.
Call the directed graph $(\A_n, \to)$ \eh{ASM graph}.
}\ed

\begin{figure}
\caption{$(\mathcal{A}_3, \to)$}
\label{a3g}
\[
\xymatrix@=5mm{
&*+{\left(\begin{array}{ccc}0  &0   &1   \\ 0 & 1  & 0  \\1  & 0  &0  \end{array}\right)}
\ar@{<-}[dl]\ar@{<-}[dr]\ar@{<-}@/^9ex/[dddd]
&\\
*+{\left(\begin{array}{ccc}0 &1   &0   \\0  &0   &1   \\ 1 & 0  &0  \end{array}\right)}
\ar@{<-}[dr]\ar@{<-}[dd]\ar@{<-}@/^9ex/[ddrr]
&&*+{\left(\begin{array}{ccc}0 &0   &1   \\ 1 &0  &0   \\0  & 1  &0  \end{array}\right)}\ar@{<-}[dl]\ar@{<-}[dd]\ar@{<-}@/_9ex/[ddll]
\\
&*+{\left(\begin{array}{ccc} 0 &1   &0   \\1  &-1   &1   \\0 &1   &0  \end{array}\right)}
\ar@{<-}[dl]\ar@{<-}[dr]
\\
*+{\left(\begin{array}{ccc} 1 & 0  &0   \\ 0 &0  &1   \\0  &1   &0  \end{array}\right)}
\ar@{<-}[dr]
&&*+{\left(\begin{array}{ccc}  0& 1  &0  \\ 1 & 0  &0   \\0  &0   &1  \end{array}\right)}
\ar@{<-}[dl]
\\
&*+{\left(\begin{array}{ccc}  1& 0  &0   \\0  &1   &0   \\ 0 &0   &1  \end{array}\right)}
&
}\]
\end{figure}
\begin{figure}
\caption{$(\wt{\mathcal{A}_3}, \to)$}
\label{wta3}
\[
\xymatrix@=5mm{
&*+{\left(\begin{array}{ccc}0  &0   &1   \\ 0 & 1  & 2  \\1  & 2  &3  \end{array}\right)}
\ar@{<-}[dl]\ar@{<-}[dr]\ar@{<-}@/^9ex/[dddd]
&\\
*+{\left(\begin{array}{ccc}0 &1   &1   \\0  &1   &2   \\ 1 & 2  &3  \end{array}\right)}
\ar@{<-}[dr]\ar@{<-}[dd]\ar@{<-}@/^9ex/[ddrr]
&&*+{\left(\begin{array}{ccc}0 &0   &1   \\ 1 &1  &2   \\1  & 2  &3  \end{array}\right)}\ar@{<-}[dl]\ar@{<-}[dd]\ar@{<-}@/_9ex/[ddll]
\\
&*+{\left(\begin{array}{ccc} 0 &1   &1   \\1  &1   &2   \\1 &2   &3  \end{array}\right)}
\ar@{<-}[dl]\ar@{<-}[dr]
\\
*+{\left(\begin{array}{ccc} 1 & 1  &1   \\ 1 &1  &2   \\1  &2   &3  \end{array}\right)}
\ar@{<-}[dr]
&&*+{\left(\begin{array}{ccc}  0& 1  &1  \\ 1 & 2  &2   \\1  &2   &3  \end{array}\right)}
\ar@{<-}[dl]
\\
&*+{\left(\begin{array}{ccc}  1& 1  &1   \\1  &2   &2   \\ 1 &2   &3  \end{array}\right)}
&
}\]
\end{figure}

It naturally induces the same directed graph structure on $\wt{\A}_n$; by abuse of language, we call it ASM graph as well.\\
As shown above, every edge in Bruhat graph is also an edge in ASM graph; see Figure \ref{a3g}. In terms of this new graph, we may characterize ASM order as follows:

\bp{The following are equivalent:
\be{\i $A\le B$ in ASM order.
\i There exists a directed path from $A$ to $B$.
}\ee
}\ep







\subsection{Key Lemma}

We defined the edge relation for two ASMs in terms of their corner sum matrices. Along this relation, what happens back to entries of the two ASMs? Key Lemma \ref{answer} below answers this question completely; it will play a key role to prove main results in Section \ref{sec5}. Before that, we take auxiliary two steps with the following lemmas.


\bl[nonpositivity]{
Let $B\in \A_n$. Suppose $R_{ij}^{kl}\in E(B)$ is given. 
Then, $b_{ik}\le 0$ and $b_{jl}\le 0$.
}\el


\bpf{
Suppose $R_{ij}^{kl}\in E(B)$. Thanks to one of the essential conditions 
$\wt{B}(i, k)=\wt{B}(i, k-1)$, we have 
\[b_{ik}=\wt{B}(i, k)+\wt{B}(i-1, k-1)-\wt{B}(i, k-1)-\wt{B}(i-1, k)
=\wt{B}(i-1, k-1)-\wt{B}(i-1, k)\le 0.
\]
Moreover, two of essential conditions $\wt{B}(j, l-1)=\wt{B}(j-1, l-1)+1$ and 
$\wt{B}(j-1, l)=\wt{B}(j-1, l-1)+1$ imply that 
\begin{align*}
b_{jl}&=\wt{B}(j, l)+\wt{B}(j-1, l-1)-\wt{B}(j, l-1)-\wt{B}(j-1, l)\\
&=\wt{B}(j, l)-\wt{B}(j-1, l-1)-2\le 0.
\end{align*}
}\epf

\bl[nonnegativity]{Let $B\in \A_n$. Suppose $R_{ij}^{kl}\in E(B)$ is given. 
Then, $b_{il}\ge 0$ and $b_{jk}\ge 0$.
}\el

\bpf{
Thanks to one of essential conditions $\wt{B}(i, l)=\wt{B}(i, l-1)+1$, we have 
\[b_{il}=\wt{B}(i, l)+\wt{B}(i-1, l-1)-\wt{B}(i, l-1)-\wt{B}(i-1, l)
=\wt{B}(i-1, l-1)-\wt{B}(i-1, l)+1\ge 0.
\]
It is similar to show that $b_{jk}\ge 0$.
}\epf




\begin{table}[htp]
\renewcommand{\arraystretch}{1.7}
\caption{16 kinds of edge relations $A\to B$ in ASM graph}
\begin{center}
\begin{tabular}{|c|c|c||c|c|c|}\hline
type&
$\left(\begin{array}{cc}b_{ik}& b_{il} \\b_{jk}  &b_{jl}  \end{array}\right)$&$\left(\begin{array}{cc}a_{ik}& a_{il} \\a_{jk}  &a_{jl}  \end{array}\right)$
&type&$\left(\begin{array}{cc}b_{ik}& b_{il} \\b_{jk}  &b_{jl}  \end{array}\right)$&$\left(\begin{array}{cc}a_{ik}& a_{il} \\a_{jk}  &a_{jl}  \end{array}\right)$\\\hline
1&$\left(\begin{array}{cc}{0} & 1 \\1  &0  \end{array}\right)$&
$\left(\begin{array}{cc}1&0  \\0  &1  \end{array}\right)$
&9&$\left(\begin{array}{cc}-1 & 1 \\1  &0  \end{array}\right)$&
$\left(\begin{array}{cc}0&0  \\0  &1  \end{array}\right)$
\\\hline
2&$\left(\begin{array}{cc}{0} & 0 \\1  &0  \end{array}\right)$&$\left(\begin{array}{cc}1&  -1\\ 0 & 1 \end{array}\right)$
&10&$\left(\begin{array}{cc}{-1} & 0 \\1  &0  \end{array}\right)$&$\left(\begin{array}{cc}0&  -1\\ 0 & 1 \end{array}\right)$
\\\hline
3&$\left(\begin{array}{cc}{0} & 1 \\0  &0  \end{array}\right)$&
$\left(\begin{array}{cc}1& 0 \\-1  & 1 \end{array}\right)$
&11&$\left(\begin{array}{cc}{-1} & 1 \\0  &0  \end{array}\right)$&
$\left(\begin{array}{cc}0& 0 \\-1  & 1 \end{array}\right)$
\\\hline
4&$\left(\begin{array}{cc}{0} & 0 \\0  &0  \end{array}\right)$&$\left(\begin{array}{cc}1&  -1\\-1  &1  \end{array}\right)$
&12&$\left(\begin{array}{cc}{-1} & 0 \\0  &0  \end{array}\right)$&$\left(\begin{array}{cc}0&  -1\\-1  &1  \end{array}\right)$
\\\hline
5&$\left(\begin{array}{cc}{0} & 1 \\1  &-1  \end{array}\right)$&
$\left(\begin{array}{cc}1& 0 \\ 0 & 0 \end{array}\right)$
&13&$\left(\begin{array}{cc}{-1} & 1 \\1  &-1  \end{array}\right)$&
$\left(\begin{array}{cc}0& 0 \\ 0 & 0 \end{array}\right)$
\\\hline
6&$\left(\begin{array}{cc}{0} & 0 \\1  &-1  \end{array}\right)$&$\left(\begin{array}{cc}1& -1 \\0  &0  \end{array}\right)$
&14&$\left(\begin{array}{cc}{-1} & 0 \\1  &-1  \end{array}\right)$&$\left(\begin{array}{cc}0& -1 \\0  &0  \end{array}\right)$
\\\hline
7&$\left(\begin{array}{cc}{0} & 1 \\0  &-1  \end{array}\right)$&$\left(\begin{array}{cc}1& 0 \\ -1 & 0 \end{array}\right)$
&15&$\left(\begin{array}{cc}{-1} & 1 \\0  &-1  \end{array}\right)$&$\left(\begin{array}{cc}0& 0 \\ -1 & 0 \end{array}\right)$
\\\hline
8&$\left(\begin{array}{cc}{0} & 0 \\0  &-1  \end{array}\right)$&$\left(\begin{array}{cc}1& -1 \\-1  &0  \end{array}\right)$
&16&$\left(\begin{array}{cc}{-1} & 0 \\0  &-1  \end{array}\right)$&$\left(\begin{array}{cc}0& -1 \\-1  &0  \end{array}\right)$
\\\hline
\end{tabular}
\end{center}
\label{d1}
\end{table}%

These two lemmas assert that each of $b_{ik}, b_{il}, b_{jk}, b_{jl}$ can take two values. In total, there are 16 cases as listed in Table \ref{d1}.

\begin{keylem}
\label{answer}
Let $B\in \A_n$ and $R_{ij}^{kl}\in E(B)$. Consider a square matrix $A$ of size $n$. Then, the following are equivalent:
\be{\i $A\bs{ij}{\to}{kl} B$.
\i 
The entries $(a_{ik}, a_{il}, a_{jk}, a_{jl})$ satisfy 
\[
\left(\begin{array}{cc}a_{ik}& a_{il} \\a_{jk}  &a_{jl}  \end{array}\right)-
\left(\begin{array}{cc}b_{ik}& b_{il} \\b_{jk}  &b_{jl}  \end{array}\right)=\left(\begin{array}{cc}1& -1 \\-1  &1  \end{array}\right)\]
 as listed in Table \ref{d1}. Moreover, if $(p, q)\not\in\{(i, k), (i, l), (j, k), (j, l)\}$, then $a_{pq}=b_{pq}$.
}\ee
\end{keylem}

\bpf{$(1)\then (2)$: Suppose $\wt{A}=\wt{B}+\wt{R}_{ij}^{kl}$ so that $\wt{A}(p, q)=\wt{B}(p, q)$ if and only if $(p, q)\not\in R_{ij}^{kl}$.
Thus, equalities 
\begin{align*}
a_{ik}&=\wt{A}(i, k)+\wt{A}(i-1, k-1)-\wt{A}(i-1, k)-\wt{A}(i, k-1) \te{ and}\\
b_{ik}&=\wt{B}(i, k)+\wt{B}(i-1, k-1)-\wt{B}(i-1, k)-\wt{B}(i, k-1)
\end{align*}
show that $a_{ik}-b_{ik}=\wt{A}(i, k)-\wt{B}(i, k)=1$ (the other six terms are gone). Similarly, 
\begin{align*}
a_{il}&=\wt{A}(i, l)+\wt{A}(i-1, l-1)-\wt{A}(i-1, l)-\wt{A}(i, l-1) \te{ and}\\
b_{il}&=\wt{B}(i, l)+\wt{B}(i-1, l-1)-\wt{B}(i-1, l)-\wt{B}(i, l-1)
\end{align*}
show that $a_{il}-b_{il}=-\wt{A}(i, l-1)+\wt{B}(i, l-1)=-1$.
In the same way, $a_{jk}-b_{jk}=-1$.
Likewise, 
\begin{align*}
a_{jl}&=\wt{A}(j, l)+\wt{A}(j-1, l-1)-\wt{A}(j-1, l)-\wt{A}(j, l-1) \te{ and}\\
b_{jl}&=\wt{B}(j, l)+\wt{B}(j-1, l-1)-\wt{B}(j-1, l)-\wt{B}(j, l-1)
\end{align*}
show that $a_{jl}-b_{jl}=\wt{A}(j-1, l-1)-\wt{B}(j-1, l-1)=1$.
For other $(p, q)$, observe that $|\{(p, q), (p-1, q-1), (p-1, q), (p, q-1)\}\cap R_{ij}^{kl}|$ is either $0$, $2$ or $4$.
If it is 0 or 4, then clearly $a_{pq}=b_{pq}$ follows.
If it is $2$, then either $p\in\{i, j\}$ or $q\in\{k, l\}$.
Here suppose $p=i$ and $q\not\in\{k, l\}$ so that 
\[a_{pq}-b_{pq}=\wt{A}(p, q)-\wt{B}(p, q)-(\wt{A}(p, q-1)-\wt{B}(p, q-1))
=1-1=0.\]
It is analogous to verify other cases.\\
$(2)\then (1)$: We can reverse most of the proof above.
}\epf

Table \ref{d1} indicates such 16 edge relations; note that only the type 1 occurs in Bruhat graphs.
It is convenient to say that a 2 by 2 minor in an ASM is \eh{interchangeable} if  it is one of the 32 patterns in the table.

\bex{
Let $B=\left(\begin{array}{ccccc}0 & 1 & 0 & 0 & 0 \\{0} & 0 & 1 & 0 & 0 \\1& {-1} & 0 & {0} & 1 \\0 & 1 & {-1} & 1 & 0 \\0 & 0 & 1 & 0 & 0\end{array}\right)$ be an ASM of size 5. Its corner sum matrix is 
$\left(\begin{array}{ccccc}0 & 1 & 1& 1 & 1 \\{0} & 1 & 2 & 2 & 2 \\1& {1} & 2 & {2} & 3 \\1& 2 & \ul{\,2\,} & 3 & 4 \\1 & 2 & 3 & 4 & 5\end{array}\right)$. 
Here the underlined part refers to $R_{45}^{34}$.
Then, we have 
\[
A=\left(\begin{array}{ccccc}0 & 1 & 0 & 0 & 0 \\{0} & 0 & 1 & 0 & 0 \\1& {-1} & 0 & {0} & 1 \\
0 & 1 & \ul{\,0\,} & \ul{\,0\,} & 0 \\0 & 0 & \ul{\,0\,} & \ul{\,1\,} & 0\end{array}\right)
\to \left(\begin{array}{ccccc}0 & 1 & 0 & 0 & 0 \\{0} & 0 & 1 & 0 & 0 \\1& {-1} & 0 & {0} & 1 \\0 & 1 & \ul{\,-1\,} & \ul{\,1\,} & 0 \\0 & 0 & \ul{\,1\,} & \ul{\,0\,} & 0\end{array}\right)=B.\]
This is type 9.
}\eex

\subsection{Essential points}

As seen in the previous example, an essential rectangle can be of size 1.

\bd{We say that an essential rectangle $R_{ij}^{kl}$ is 
an \emph{essential point} 
if $j=i+1$ and $l=k+1$ (so that $|R_{ij}^{kl}|=1$).
}\ed

\brk{
Here, we have a specific reason to coin the term ``essential point"; 
Fulton \cite{fulton} defined \eh{essential sets} for permutations as follows:
\begin{align*}
\Ess(w)=\{(i, j)\in [n-1]^2\mid i<w^{-1}(j), j<w(i), w(i+1)\le j, w^{-1}(j+1)\le i\}.
\end{align*}
We may rephrase these four conditions in terms of corner sum matrices:
For each $(i, j)\in [n-1]^2$, the following equivalences hold (as easily checked):
\[\left.\begin{array}{rlll}(1)&i<w^{-1}(j)  &\iff   &\wt{w}(i-1, j)=\wt{w}(i, j).   \\
 (2)& j<w(i) & \iff  & \wt{w}(i, j-1)=\wt{w}(i, j).  \\ (3)&w(i+1)\le j  &\iff   & \wt{w}(i+1, j)=\wt{w}(i, j)+1.  \\
 (4)& w^{-1}(j+1)\le i & \iff  & \wt{w}(i, j+1)=\wt{w}(i, j)+1. \end{array}\right.\]
Thus, $(i, j)$ is an essential point of $w$ if and only if $(i, j)$ is an element of  $\Ess(w)$.}\erk

As a consequence of Key Lemma, there is a one-to-one correspondence between essential points of $B$ and ASMs covered by $B$.
Hence every covering relation in ASM order is an edge relation of ASM graph.







Define a permutation $w$ to be \eh{bigrassmannian} if there exists a unique pair $(i, j)\in [n-1]^2$ with $w^{-1}(i)>w^{-1}(i+1)$ and $w(j)>w(j+1)$.

\bp{For $A\in \A_n$, the following are equivalent:
\be{\i $A$ is a bigrassmannian permutation.
\i $A$ has exactly one essential point.
}\ee
}\ep

\bpf{(Sketch)
Both are equivalent to what we call \eh{join-irreducibility};
see Lascoux-Sch\"{u}tzenberger \cite{ls} for details of equivalence of bigrassmannian and join-irreducibility. Recall from the theory of finite distributive lattices \cite{reading} that an element is join-irreducible in such a lattice if and only if it covers exactly one element.
}\epf

For example, 
$\wt{1342}=\left(\begin{array}{cccc}1 & 1 & 1 & 1 \\1 & 1 & 2 & 2 \\1 & \ul{\,1\,} & 2 & 3 \\1 & 2 & 3 & 4\end{array}\right)
$ has exactly one essential point so that 1342 is bigrassmannian.

\bp{\hfill
\be{\i (Chain Property) If $A\le B$, then there exists a directed path 
\[
A\to A_1\to A_2\to \dots \to A_k=B\]
 such that $\beta(A_i)-\beta(A_{i-1})=1$ for all $i$.
\i For each $A\in \A_n$, we have 
\[\b(A)=|\{B\in \A_n\mid B\le A \ph{a}\tn{ and } B \ph{a}\tn{ is bigrassmannian}\}|.\]
}\ee
}\ep

\bpf{(Sketch) As Reading reviewed \cite{reading}, $(\A_n, \le)$ is (isomorphic to) a finite distributive lattice graded by $|\{B\in \A_n\mid B\le A \ph{a}\tn{ and } B \ph{a}\tn{ is join-irreducible}\}|$. Since $\b(e)=0$ ($e$ the minimum element) and $\b$ increases by one along every covering relation, this function must coincide with $\b$. As a result, these two assertions follow.
}\epf



For this reason, we call $\b$ \eh{bigrassmannian statistics}. We will show more explicit formulas for $\b$ in Section \ref{sec5}.











\section{Total nonnegativity and (SFL) property}\label{sec4}

Toward our main result, we now need key ideas: total nonnegativity and subtraction-free Laurent (SFL) property. Although these are classical topics in  applications of Linear Algebra (as Ando \cite{ando}), here let us review precise definitions of such ideas.

\subsection{Total nonnegativity}

Let $A$ be a real $n$ by $n$ matrix.
\bd{We say that $A=(a_{ij})$ is \emph{totally nonnegative} (TNN) if 
the determinant for every square submatrix of $A$ is nonnegative.
}\ed

\brk{
Some authors use the term ``totally positive" to mean the same thing. Here we followed Drake-Gerrish-Skandera \cite{dgs1,dgs2}.
}\erk

Let $x_{11}, \dots, x_{nn}$ be commutative variables and  $f(x_{11}, \cd, x_{nn})$ a real polynomial. When no confusion arises, we simply write $f(x)$ to mean the polynomial $f(x_{11}, \dots, x_{nn})$. Similarly, for a real matrix $A=(a_{ij})$, we write $f(A)$ to mean the real number $f(a_{11}, \dots, a_{nn})$.

\bd{We say that a polynomial $f(x)$ is \emph{totally nonnegative} (TNN) if whenever $A$ is a TNN matrix of size $n$,  then $f(A)\ge 0$.
}\ed
\brk{In particular, if this is the case, then we have $a_{ij}\ge 0$ for every $(i, j)$ because $a_{ij}$ is itself the determinant of a 1 by 1 submatrix.}\erk
\bd{
Given $u\in S_n$, let $\x^u$ denote the monomial $x_{1u(1)}\cdots x_{nu(n)}$.
We call it the \emph{permutation monomial} for $u$.}\ed
\bex{Let 
$u=\left(\begin{array}{ccc} 0 & 1  &0   \\1  & 0  &0   \\0  & 0  &1  \end{array}\right)$ and 
$v=\left(\begin{array}{ccc} 0 & 0  & 1  \\1  & 0  & 0  \\0  &1   &0  \end{array}\right)$.  
Then 
\[
\x^u-\x^v=x_{12}x_{21}x_{33}-x_{13}x_{21}x_{32}\]
 is TNN since 
we have the inequality
\[a_{12}a_{21}a_{33}-a_{13}a_{21}a_{32}=a_{21}\left|\begin{array}{cc}a_{12}  & a_{13}  \\a_{32}  & a_{33} \end{array}\right|\ge 0\]
for all TNN matrices $A=(a_{ij})$.
}\eex

Now we extend total nonnegativity for ASMs.
As above, let $x_{11}, \dots, x_{nn}$ be commutative variables. 
For our purpose, consider a rational function $g(x)=g(x_{11}, \cd, x_{nn})$ rather than a polynomial.
\bd{We say that a rational function $g(x)$ is \emph{totally nonnegative} (TNN) if whenever $A$ is a TNN matrix of size $n$ and moreover $g(A)$ is defined, then $g(A)\ge 0$.
}\ed

If $g(x)$ is indeed a polynomial, then this definition coincides with the total nonnegativity above.




\bd{For each $A\in \A_n$, 
introduce the \emph{ASM (Laurent) monomial} 
\[
\x^A:=\prod_{i, j=1}^{n}x_{ij}^{a_{ij}}.\]
}\ed

Apparently, this idea includes permutation monomials.







\bex{
Let $B=\left(\begin{array}{ccc} 0 & 1  &0   \\1  & 0  &0   \\0  & 0  &1  \end{array}\right)$ and 
$C=\left(\begin{array}{ccc} 0 & 1  & 0  \\1  & -1  & 1  \\0  &1   &0  \end{array}\right)$.


Then $g(x)=\x^B-\x^C=x_{12}x_{21}x_{33}-x_{12}x_{21}x_{22}^{-1}x_{23}x_{32}$ is TNN since we have the inequality
\[g(A)=a_{12}a_{21}a_{33}-a_{12}a_{21}a_{22}^{-1}a_{23}a_{32}
=a_{12}a_{21}\fr{\left|\begin{array}{cc}  a_{22}&a_{23}   \\a_{32}  &a_{33}  \end{array}\right|}{a_{22}}\ge 0
\]
for all TNN matrices $A=(a_{ij})$ such that $a_{22}\ne 0$.

}\eex

This example suggests the following consequence of Key Lemma. 
If $A\to B$, then there exists a unique $(i, j, k, l)\in [n]^4$ such that 
\[\{(p, q)\in[n]^2\mid a_{pq}\ne b_{pq}\}=\{(i, k), (i, l), (j, k), (j, l)\}.\]
It leads to a decomposition of a difference of ASM monomials: 
Set \[\x^{AB}:= \prod_{a_{pq}= b_{pq}} x_{pq}^{a_{pq}} \mbox{
and } \x^{E(A, B)}:= \prod_{a_{pq}\ne b_{pq}} x_{pq}^{a_{pq}}-\prod_{a_{pq}\ne b_{pq}}x_{pq}^{b_{pq}}.\]
Clearly, the latter corresponds to interchangeable entries of $A$ and $B$. These two rational functions give the decomposition $\x^A-\x^B=\x^{AB}\x^{E(A, B)}$.
Observe that, in any case, $\x^{E(A, B)}$ is a product of $\left|\begin{array}{cc}x_{ik}& x_{il} \\x_{jk}  &x_{jl}  \end{array}\right|$ and a Laurent monomial in these four variables as 
\begin{align*}
\x^{E(A, B)}&=
x_{ik}^{a_{ik}}x_{il}^{a_{il}}x_{jk}^{a_{jk}}x_{jl}^{a_{jl}}-
x_{ik}^{b_{ik}}x_{il}^{b_{il}}x_{jk}^{b_{jk}}x_{jl}^{b_{jl}}\\
&=x_{ik}^{a_{ik}}x_{il}^{a_{il}}x_{jk}^{a_{jk}}x_{jl}^{a_{jl}}-
x_{ik}^{a_{ik}-1}x_{il}^{a_{il}+1}x_{jk}^{a_{jk}+1}x_{jl}^{a_{jl}-1}\\
&=x_{ik}^{a_{ik}}x_{il}^{a_{il}}x_{jk}^{a_{jk}}x_{jl}^{a_{jl}}\fr{\left|\begin{array}{cc}x_{ik}& x_{il} \\x_{jk}  &x_{jl}  \end{array}\right|}{x_{ik}x_{jl}}.
\end{align*}

\subsection{(SFL) property}

Let $f(x)$ be a real polynomial. 
\bd{We say that $f(x)$ has \emph{Subtraction-Free Rational} (SFR) property if $f(x)$ has a rational expression in minors of the matrix $x=(x_{ij})_{i, j=1}^n$ such that its denominator and numerator do not contain any subtraction.
Also say that $f(x)$ has \emph{Subtraction-Free Laurent} (SFL) property if $f(x)$ has (SFR) property with a rational expression such that its denominator is a monomial in minors of $x$.}\ed
We could define these properties for rational functions of $x_{11}, \dots, x_{nn} $ in the exactly same way. For example, $g(x)=x_{12}x_{21}x_{33}-x_{12}x_{21}x_{22}^{-1}x_{23}x_{32}$ has (SFR) and (SFL) properties as mentioned above.

\subsection{Drake-Gerrish-Skandera's characterizations} 

In the last two subsections, we reviewed two properties on polynomials.
What is the relation between (TNN), (SFL) properties and Bruhat order?
Drake-Gerrish-Skandera \cite{dgs1,dgs2} established the following equivalence:
\bft{
Let $u, v\in S_n$. Then the following are equivalent:
\be{\i $u\le v$ in Bruhat order.
\i $\x^u-\x^v$ is TNN.
\i $\x^u-\x^v$ has (SFL) property.
}\ee
}\eft

In the next section, we generalize this result as Theorem \ref{qth}.




\section{Main results}\label{sec5}

In this section, we give main results as Theorems \ref{bff}, \ref{qth} and \ref{mth3} with proofs.

\subsection{Bigrassmannian statistic}
\label{bsdt}

A bigrassmannian statistic is a meaningful number counting entries of corner sum matrices as the rank function of the finite distributive lattice. We now show a simple and new enumerative formula on entries of ASMs; 
this generalizes the author's formula \cite{kob}.
the directed graph structure plays a role for a proof.






\bt{For each $B\in \A_n$, we have \label{bff}
\[\b(B)=\sum_{i, j=1}^n\fr{ (i-j)^2}{2}b_{ij}.\]
}\et

\bpf{Let $\al(B)$ be the sum on the right hand side.
We will show that $\beta(B)=\al(B)$ by induction on $\b(B)$.
If $\b(B)=0$, then $B=e=(\delta_{ij})$ so that 
$\al(B)=\sum \fr{ (i-j)^2}{2}\delta_{ij}=0$.
Suppose $\b(B)>0$.
Choose $A\in\A_n$ such that $A\to B$, say $A\bs{ij}{\to}{kl} B$ so that 
\[
\b(B)-\b(A)=|R_{ij}^{kl}|=(j-i)(l-k).\]
It is now enough to show $\al(B)-\al(A)=(j-i)(l-k)$, that is, $\al$ satisfies the same recursion (which further shows that $\al(B)$ is an integer for all $B$).
Four entries $(a_{ik}, a_{il}, a_{jk}, a_{jl})$ must be one of the 16 cases listed in Table \ref{d1}.
 It follows, in any case, that
\begin{align*}
\al(B)-\al(A)&=\sum_{p, q=1}^n\fr{(p-q)^2}{2}(b_{pq}-a_{pq})\\
&=-\fr{(i-k)^2}{2}+\fr{(i-l)^2}{2}+\fr{(j-k)^2}{2}-\fr{(j-l)^2}{2}=(j-i)(l-k).
\end{align*}



}\epf

\begin{cor}
$\beta(w)=\di\sum_{i=1}^{n}\frac{1}{\,2\,}(i-w(i))^{2}$ for $w\in S_{n}$.
\end{cor}

\begin{proof}
Use the theorem. For $B=w$, we have 
$b_{ij}\ne 0$ if and only if 
$b_{ij}=1$ and $j=w(i)$.
\end{proof}
\begin{ex}
$\beta(4312)=
\di\frac{1}{\,2\,}\left((1-4)^{2}+(2-3)^{2}+(3-1)^{2}+(4-2)^{2}\right)=9.$
\end{ex}

%
%
%

\subsection{($q$TNN) and ($q$SFL) properties}

We next introduce a $q$-analog of (TNN) and (SFL) properties. Motivated by Theorem \ref{bff}, we will consider a $q$-analog of our variables $x_{11}, \dots, x_{nn}$.
From now on, regard $q$ as a variable taking positive real numbers so that ``$q^{1/2}$" makes sense.
For each $(i, j)$, let $x_{ij, q}:=q^{(i-j)^2/2} x_{ij}$ and call $\{x_{ij, q}\}$ \emph{$q$-variables}.
Given a matrix $x=(x_{ij})$, let $x_q=(x_{ij, q})$ denote its $q$-analog.
Further, let $f(x_q)$ mean the polynomial $f(x_{11, q}, \dots,x_{nn, q})$ in $x_{ij}$ and $q$. In particular, the \emph{ASM (Laurent) $q$-monomial} for an ASM  $A$ is  
\[
\x_q^{A}:=\prod_{i,j} (x_{ij, q})^{a_{ij}}\,(=q^{\b(A)}\x^A).\] 
For example, if $A=\left(\begin{array}{ccc} 0 & 1  & 0  \\1  & -1  & 1  \\0  &1   &0  \end{array}\right)$, then 
\begin{quote}
$A_q=\left(\begin{array}{ccc} 0 & q^{1/2}  & 0  \\q^{1/2}  & -1  & q^{1/2}  \\0  &q^{1/2}   &0  \end{array}\right)$
and $\x_q^{A}=q^2x_{12}x_{21}x_{22}^{-1}x_{23}x_{32}$.
\end{quote}


\bd{Fix a positive real number $q_0$. Say a square matrix $A$ is \emph{locally TNN at $q_0$} if all minors of $A_{q_0}$ are nonnegative.}\ed

\brk{Let us make sure that ``$A$ is \emph{locally TNN at $1$}" is equivalent to saying ``$A$ is TNN" as defined earlier.}\erk

\bd{We say that ``$A$ is \emph{$q$TNN}" if it is locally TNN at $q$ for all $q>0$.
}\ed

We next introduce a $q$-analog of (extended) total nonnegativity.
Let $g(x)$ be a rational function in $x_{11}, \dots, x_{nn}$ as before.
\bd{Say $g(x)$ is \emph{locally TNN at $q_0$} if whenever $A$ is locally TNN at $q_0$ and moreover $g(A_{q_0})$ is defined, then $g(A_{q_0})\ge 0$.
Say ``$g(x)$ is \emph{qTNN}" if it is locally TNN at $q$ for all $q>0$. 
}\ed










List all minors of $x$ as $\Delta=\{\Delta_1(x), \dots, \Delta_m(x)\}$.

\bd{Say a rational function $g(x)$ in $x_{11}, \dots, x_{nn}$ has \emph{(qSFL) property} if there exist $F(x), G(x)\in \rr[x]$ such that \label{qsfl}
\be{\item $g(x)=F(x)/G(x)$,
\item $F(x)=\sum c_{i_1\cdots i_k}\Delta_{i_1}(x)\cdots\Delta_{i_k}(x)$ with $c_{i_1\cdots i_k}$ nonnegative integers, i.e., a subtraction-free polynomial in minors of $x$,
\item $G(x)=\prod_j \Delta_{j}(x)^{d_j}$ with $d_j$ nonnegative integers, i.e., a monomial in minors of $x$ and 
\item $g(x_q)=F(x_q)/G(x_q)\in \rr(x)[q]$, i.e., $g(x_q)$ is a polynomial in $q$.
}\ee

}\ed

Observe that if $g_1(x)$ and $g_2(x)$ have ($q$SFL) property, then so does $g_1+g_2$.

\subsection{Characterizations of ASM order}
\bt{Let $A, B\in \A_n$. Then the following are equivalent:\label{qth}
\be{\item $A\le B$ in ASM order.
\item $\x^A-\x^B$ is $q$TNN.
\item $\x^A-\x^B$ has ($q$SFL) property.
}\ee
}\et
We prove $(1)\then (3) \then (2) \then (1)$.


\bpf{
(1) $\then$ (3): The assertion is obvious for $A=B$. Let us suppose $A<B$. 
We first deal with the case $A\to B$, say $A\bs{ij}{\to}{kl} B$; this relation belongs to precisely one of 16 cases in Table \ref{d1}.
Recall that $\x^A-\x^B=\x^{AB}\x^{E(A, B)}$ with $\x^{AB}$ a Laurent monomial in $x_{11}, \dots, x_{nn}$ and $\x^{E(A, B)}$ a subtraction-free Laurent rational expression in minors of $x$. Hence $\x^A-\x^B$ has (SFL) property.
Moreover, $\x_q^{A}-\x_q^{B}=q^{\beta(A)}(\x^A-q^{(j-i)(l-k)}\x^B)$ is certainly a polynomial in $q$ so that we proved  ($q$SFL) property for $\x^A-\x^B$.
Suppose next $A<B$. By another interpretation of ASM order with ASM graph, we can find a directed path \[
A=A_0\to A_1\to \dots\to A_N=B.\] Now write 
\[\x^A-\x^B=(\x^{A_0}-\x^{A_1})+(\x^{A_1}-\x^{A_2})+\cdots +(\x^{A_{N-1}}-\x^{A_N}).\]
This is a sum of rational functions all of which have ($q$SFL) property. Hence so does $\x^A-\x^B$.\\
(3) $\then$ (2): 
Suppose $g(x)=\x^A-\x^B$ has ($q$SFL) property, say $g(x)=F(x)/G(x)$ as in Definition \ref{qsfl}.
We want to show that $g(x)$ is $q$TNN. For this purpose, we first verify a local condition: choose $q_0>0$ and let $A'$ be a locally TNN matrix at $q_0$ such that $G(A'_{q_0})\ne 0$. Then $g(A'_{q_0})=F(A'_{q_0})/G(A'_{q_0})\ge 0$ because each term in the sum $F(A'_{q_0})$ and each factor in the product $G(A'_{q_0})$ are nonnegative. Thus $g(x)$ is locally TNN at $q_0$. 
This is true for all $q_0>0$. Hence $g(x)$ is $q$TNN. \\
$(2)\then (1)$: This proof is almost same to Drake-Gerrish-Skandera \cite{dgs1,dgs2}. Nonetheless, we repeat it here.
Suppose $A\not\le B$. We may choose indices $k, l\in [n]$ such that $\wt{A}(k, l)<\wt{B}(k, l)$.
Now define the matrix $A'=(a_{ij}')$ by $a_{ij}'=\begin{cases}2 & i\le k \mbox{ and } j\le l\\
1&\mbox{ otherwise.}
\end{cases}$
It is easy to see that $A'$ is TNN since all square submatrices of $A'$ have determinant $0, 1$, or $2$. Now $x_{ij}=a_{ij}'$ yields
\begin{align*}
\x^A-\x^B\Bigr|_{x_{ij}=a'_{ij}}&=\pr_{i, j=1}^n(a_{ij}')^{a_{ij}}-\pr_{i, j=1}^n(a_{ij}')^{b_{ij}}\\
&=\pr_{i\le k, j\le l}2^{a_{ij}}-\pr_{i\le k, j\le l}2^{b_{ij}}
=2^{\wt{A}(k, l)}-2^{\wt{B}(k, l)}<0.
\end{align*}

Thus, $\x^A-\x^B$ is not TNN, i.e., $\x^A-\x^B$ is not locally TNN at $1$.
Hence $\x^A-\x^B$ is not $q$TNN.
 }\epf



\subsection{Corollaries}

We observe several corollaries. First, $q=1$ in Theorem \ref{qth} recovers this equivalence:
\bc{
Let $A, B\in \A_n$. Then the following are equivalent:
\be{\item $A\le B$ in ASM order.
\item $\x^A-\x^B$ is TNN.
\item $\x^A-\x^B$ has (SFL) property.
}\ee
}\ec


\bex{
Let 
$A=\left(\begin{array}{ccccc}0 & 1 & 0 & 0 & 0 \\1 & -1 & 1 & 0 & 0 \\0& 1 & -1 & 0& 1 \\0 & 0 & 0 & 1 & 0 \\0 & 0 & 1 & 0 & 0\end{array}\right)$, 
$B=\left(\begin{array}{ccccc}0 & 1 & 0 & 0 & 0 \\1 & -1 & 1 & 0 & 0 \\0& 0 & 0 & 0 & 1 \\0 & 1 & -1 & 1 & 0 \\0 & 0 & 1 & 0 & 0\end{array}\right)$ and 
$C=\left(\begin{array}{ccccc}0 & 1 & 0 & 0 & 0 \\0 & 0 & 1 & 0 & 0 \\1& -1 & 0 & 0 & 1 \\0 & 1 & -1 & 1 & 0 \\0 & 0 & 1 & 0 & 0\end{array}\right)$.
Since $A\to B\to C$, $\x^{A}-\x^{C}$ is TNN and has (SFL) property:
\[\x^{A}-\x^{C}=(\x^{A}-\x^{B})+(\x^{B}-\x^{C})
=
\]
\[\fr{x_{12}x_{21}x_{23}x_{35}x_{44}x_{53}}{x_{22}}\left(
\fr{x_{32}x_{43}-x_{42}x_{33}}{x_{43}x_{33}}
\right)
+\fr{x_{12}x_{23}x_{35}x_{42}x_{44}x_{53}}{x_{43}}\left(
\fr{x_{21}x_{32}-x_{31}x_{22}}{x_{32}x_{22}}
\right)
\]
\[=\fr{x_{12}x_{23}x_{35}x_{44}x_{53}
}{x_{22}x_{32}x_{33}x_{43}}
\left(x_{21}x_{32}\left|\begin{array}{cc} x_{32} & x_{33}  \\x_{42}  & x_{43} \end{array}\right|
+
x_{33}x_{42}\left|\begin{array}{cc} x_{21} & x_{22}  \\x_{31}  &x_{32}  \end{array}\right|
\right).
\]
As expected, this is a subtraction-free Laurent rational expression in minors of $x$.
It follows that 
\[\x_q^{A}-\x_q^{C}=(\x_q^{A}-\x_q^{B})+(\x_q^{B}-\x_q^{C})
\]
\[=\left.\fr{x_{12}x_{23}x_{35}x_{44}x_{53}
}{x_{22}x_{32}x_{33}x_{43}}
\left(x_{21}x_{32}\left|\begin{array}{cc} x_{32} & x_{33}  \\x_{42}  & x_{43} \end{array}\right|
+
x_{33}x_{42}\left|\begin{array}{cc} x_{21} & x_{22}  \\x_{31}  &x_{32}  \end{array}\right|
\right)\right\rvert_{x_{ij}\mapsto x_{ij, q}}.
\]

This is a subtraction-free Laurent rational expression in minors of $x_q$; moreover, $\b(A)=\fr{1}{2}+\fr{1}{2}+\fr{1}{2}+\fr{1}{2}+\fr{4}{2}+\fr{4}{2}=$ 6 so that $\x_q^{A}-\x_q^{C}=(q^6\x^A-q^7\x^B)+(q^7\x^B-q^8\x^C)$, certainly a polynomial in $q$.
}\eex


Here we record some consequence of this example (motivated by recent developments on algebraic combinatorics such as total positivity \cite{fz1}, and cluster algebras \cite{fz2}); for convenience, we prepare several words. Let us say that a Laurent monomial $\prod_{i, j}x_{ij}^{a_{ij}}$ is \emph{almost positive} if $a_{ij} \ge -1$ for all $i, j$. Say a minor of a matrix is \emph{small} if its size is $1$ or $2$; it is \emph{solid} if its rows and columns are consecutive. 

\bc{
If $A<B$, then $\x^A-\x^B$ has a rational expression as the product $L(x)\times M(x)$ such that $L(x)$ is an almost positive Laurent monomial in $x_{11}, \dots, x_{nn}$ and $M(x)$ is a subtraction-free polynomial in only small solid minors of $x$ (without a constant term). 
}\ec

\bpf{By Chain Property, there exists a directed path $A\to A_1\to A_2\to \cdots\to A_k=B$ such that $\b(A_i)-\b(A_{i-1})=1$. As seen from Key Lemma, each $\x^{A_i}-\x^{A_{i+1}}$ is a product of an almost positive Laurent monomial and a subtraction-free polynomial in only small solid minors without a constant term. Now regarding $\x^A-\x^B$ as a sum of such, find its rational expression with choosing a common denominator. Thus, we obtain the desired expression.
}\epf

\subsection{Signed bigrassmannian statistics}
\label{sbs}
\begin{table}[htp]
{\renewcommand{\arraystretch}{1.5}
\caption{Permutation statistics}
\label{perms}
\begin{center}
\begin{tabular}{|c|c|c|c|}\hline
&Mahonian &Eulerian& Bigrassmannian\\\hline
unsigned&$q$-factorial &Eulerian polynomial&Unknown \\\hline
signed&Wachs \cite{wachs}&D\'{e}sarm\'{e}nien-Foata \cite{df}&Theorem \ref{mth3}\\\hline
\end{tabular}
\end{center}
}\end{table}%

Permutation statistics is one of important topics in combinatorics on the symmetric groups. In particular, \eh{Mahonian} and \eh{Eulerian} are well-known examples (Table \ref{perms}).
More recently, there are some work on \emph{signed Mahonian} and \eh{signed Eulerian statistics} as Wachs \cite{wachs} and D\'{e}sarm\'{e}nien-Foata \cite{df}. As one subsequent idea of their work, here we introduce \eh{signed bigrassmannian statistics}.\\
The \emph{inversion number} $\ell(w)$ for $w\in S_{n}$ is 
\[
|\{(i, j)\in [n]^2\mid i<j \mb{ and } w(i)>w(j)\}|.\] The \eh{sign} of $w$ is $(-1)^{\ell(w)}$ as often appears in the context of determinants. Now recall that $\b(w)$ gives a nonnegative integer $|\{v\in S_n\mid v\le w \mb{ and } v \mb{ bigrassmannian} \}|$ for each permutation $w$. With these notions, let us  introduce a new kind of permutation statistics: 
\bd{Define \emph{signed bigrassmannian statistics} (or \emph{signed bigrassmannian polynomial}) over $S_n$ by 
\[B_n(q)=\sum_{w\in S_n}(-1)^{\ell(w)}q^{\b(w)}.\]}\ed
For example, $B_1(q)=1, B_2(q)=1-q$ and 
$B_3(q)=1-2q+2q^3-q^4$ (missing a $q^2$ term; see Figure \ref{a3}).




\bt[Signed bigrassmannian statistics]{\label{mth3}For all $n\ge 1$, we have 
\[B_n(q)=\prod_{k=1}^{n-1}(1-q^k)^{n-k}.
\]
}\et
The idea of our proof is to show the recursion $B_n(q)=\fr{\di {\strut B_{n-1}(q)^2} (1-q^{n-1})}{\di \strut B_{n-2}(q)}$ (which is not so obvious from the definition of $B_n(q)$). We derive this equation from a series of the lemmas below. Here, we confirm our setting: The notation $|\phantom{A}|$ simply denotes the determinant. Let $A=(a_{ij})$ be an $n$ by $n$ matrix with $n\ge 2$. We formally define the determinant of the empty ($0$ by $0$) matrix is $1$.
\bl[Dodgson's condensation]{
Let $A^i_j$ denote the submatrix obtained by deleting $i$-th row and $j$-th column from $A$. Then, we have
\[|A|=\frac{|A_1^1||A_n^n|-|A_n^1||A_1^n|}{|A_{1n}^{1n}|} 
\]
provided $|A_{1n}^{1n}|\ne 0$.
}\el
\bpf{See Bressoud \cite[p.112--113]{bressoud}.
}\epf

Next, we consider a $q$-analog of this formula.

\bl[a $q$-analog of Dodgson's condensation]{
With the same notation above, we have 
\[|A_q|=\frac{|(A_1^1)_q||(A_n^n)_q|-q^{n-1}|(A_n^1)_q||(A_1^n)_q|}{|(A_{1n}^{1n})_q|} 
\]
provided $|(A_{1n}^{1n})_q|\ne 0$.
}\el

\bpf{
Apply Dodgson's condensation to $A=A_q$:
\[|A_q|=\frac{|(A_q)_1^1||(A_q)_n^n|-|(A_q)_n^1||(A_q)_1^n|}{|(A_q)_{1n}^{1n}|} 
\]
We evaluate these five determinants on the right hand side.
\be{
\i $|(A_q)_{1}^{1}|=|(q^{(i-j)^2/2}a_{ij})_{i, j=2}^{n}|
=|(q^{(i-j)^2/2}a_{i+1, j+1})_{i, j=1}^{n-1}|
=|(A_{1}^{1})_q|$.
\i It is similar to show that $(A_{n}^{n})_q=(A_q)_{n}^{n}$.
\i Using the properties of determinants, we have 
\begin{align*}
|(A_q)_n^1|&=|(q^{(i-j)^2/2}a_{ij})_{i=2, j=1}^{n, n-1}|\\
&=|(q^{(i+1-j)^2/2}a_{i+1, j})_{i, j=1}^{n-1}|\\
&=|(q^{((i-j)^2-2(i-j)+1)/2}a_{i+1, j})_{i, j=1}^{n-1}|\\
&=q^{-\sum_{}i+\sum j}|(q^{(i-j)^2/2}q^{1/2}a_{i+1, j})_{i, j=1}^{n-1}|\\
&=q^{(n-1)/2}|(q^{(i-j)^2/2}a_{i+1, j})_{i, j=1}^{n-1}|\\
&=q^{(n-1)/2}|(A_n^1)_q|.
\end{align*}
\i It is similar to show $|(A_q)_1^n|=q^{(n-1)/2}|(A_1^n)_q|$ by symmetry of rows and columns.
\i 
$|(A_q)_{1n}^{1n}|=|(q^{(i-j)^2/2}a_{ij})_{i, j=2}^{n-1}|
=|(q^{(i-j)^2/2}a_{i+1, j+1})_{i, j=1}^{n-2}|=|(A_{1n}^{1n})_q|$.
}\ee
}\epf

\bl[determinantal expression]{
Consider the matrix $A=(a_{ij})$ with $a_{ij}=1$ for all $i, j\in[n]$.
Then, $\det(A_q)=B_n(q)$. Moreover, 
$|(A_1^1)_q|=|(A_n^n)_q|=|(A_n^1)_q|=|(A_1^n)_q|=B_{n-1}(q)$
and $|(A_{1n}^{1n})_q|=B_{n-2}(q)$.
}\el

\bpf{
$\det(A_q)=\sum_{w\in S_n}(-1)^{\ell(w)}\prod_{i=1}^nq^{(i-w(i))^2/2}
=\sum_{w\in S_n}(-1)^{\ell(w)}q^{\beta(w)}=B_n(q)$.
In the same way, we can prove the other results for permutation statistics over $S_{n-1}$ and $S_{n-2}$. 
}\epf

\bpf[Proof of Theorem \ref{mth3}]{
Clearly, $B_1(q)=1$ and $B_2(q)=1-q$ are valid. 
Suppose $n\ge 3$. Apply Dodgson's condensation to $A_q$. With $a_{ij}=1$ for all $i, j$, we get
\begin{align*}
B_n(q)&=\fr{\di {\strut B_{n-1}(q)B_{n-1}(q)}- q^{(n-1)/2}B_{n-1}(q)q^{(n-1)/2}B_{n-1}(q)}{\di \strut B_{n-2}(q)}
\\
&=\fr{\di {\strut B_{n-1}(q)^2} (1-q^{n-1})}{\di \strut B_{n-2}(q)}.
\end{align*}
By induction, we conclude that 
\begin{align*}
B_n(q)&=\fr{\di {\strut B_{n-1}(q)^2} (1-q^{n-1})}{\di \strut B_{n-2}(q)}\\
&=\fr{\di \lef(\prod_{k=1}^{n-2}(1-q^k)^{n-1-k}\ri)^2}{\di \prod_{k=1}^{n-3}(1-q^k)^{n-2-k}}\,(1-q^{n-1})=
\prod_{k=1}^{n-1}(1-q^k)^{n-k}.
\end{align*}

}\epf

\bex{Observe that 
\begin{align*}
B_3(q)&=\det\left(\begin{array}{ccc} 1 &1&1  \\1 & 1  & 1  \\1  &1   &1  \end{array}\right)_q=
\det\left(\begin{array}{ccc} 1 &q^{1/2}   & q^{4/2}  \\q^{1/2}  & 1  & q^{1/2}  \\q^{4/2}  &q^{1/2}   &1  \end{array}\right)\\
&=(1-q)^2(1-q^2)=1-2q+2q^3-q^4.
\end{align*}
Let us check $B_4(q)=1-3q+q^2+4q^3-2q^4-2q^5-2q^6+4q^7+q^8-3q^9+q^{10}.$
We can see this directly from Table \ref{tab}.
Indeed, our results verify this statistics as follows:
\begin{align*}
B_4(q)&=\det\left(\begin{array}{cccc} 1 &  1 &  1 & 1  \\ 1 &  1 &  1 & 1  \\1  & 1  & 1  & 1  \\ 1 & 1  & 1  & 1 \end{array}\right)_q
=\det\left(\begin{array}{cccc} 1 & q^{1/2}
  & q^{4/2}
  & q^{9/2}
  \\q^{1/2}
  &  1 & q^{1/2}
  & q^{4/2}
  \\q^{4/2}
  &   q^{1/2}
&  1 & q^{1/2}
  \\ q^{9/2}
 & q^{4/2}
  &  q^{1/2}
 & 1 \end{array}\right)\\
 &=\fr{((1-q)^2(1-q^2))^2}{1-q}(1-q^3)\\
&=(1-q)^3(1-q^2)^2(1-q^3)\\
&=1-3q+q^2+4q^3-2q^4-2q^5-2q^6+4q^7+q^8-3q^9+q^{10}.
\end{align*}
}\eex


\begin{table}
{\renewcommand{\arraystretch}{1.5}
\caption{signed bigrassmannian statistics over $S_4$}
\label{tab}
\begin{center}
\begin{tabular}{|c|c|c|c|c|c|c|c|c|c|c|c|c|c|c|c|}\hline
&sign&$\beta$&&sign&$\beta$&&sign&$\beta$&&sign&$\beta$\\\hline
1234&$+$&0&2134&$-$&1&3124&$+$&3&4123&$-$&6\\\hline
1243&$-$&1&2143&$+$&2&3142&$-$&5&4132&$+$&7\\\hline
1324&$-$&1&2314&$+$&3&3214&$-$&4&4213&$+$&7\\\hline
1342&$+$&3&2341&$-$&6&3241&$+$&7&4231&$-$&9\\\hline
1423&$+$&3&2413&$-$&5&3412&$+$&8&4312&$-$&9\\\hline
1432&$-$&4&2431&$+$&7&3421&$-$&9&4321&$+$&10\\\hline
\end{tabular}
\end{center}
\label{default}}
\end{table}%


\section{Concluding remarks}\label{sec6}

 
In this article, we introduced a new directed graph structure (\eh{ASM graph}) into alternating sign matrices. This generalizes  Bruhat graph whose edge relation is defined by transpositions and length functions. The key idea was to consider entries of corner sum matrices rather than entries of ASMs.\\
We established subsequent results of Drake-Gerrish-Skandera \cite{dgs1,dgs2} on equivalent characterizations of Bruhat order in two ways; from permutations to ASMs; $q$-analogs with respect to the bigrassmannian statistic $\b$. 
As a by-product, we found formulas for \eh{signed bigrassmannian statistic} with a determinantal expression and Dodgson's condensation. \\
We end with several ideas for our subsequent research.
\be{\item Drake-Gerrish-Skandera proved in fact more \cite{dgs1,dgs2};  Bruhat order is equivalent to the \emph{monomial nonnegativity} (MNN) as well as the \emph{Schur nonnegativity} (SNN). Can we establish some similar results on such properties from our viewpoints such as ASMs and the $q$-analog?
\i 
Recall that we made use of a $q$-analog of the determinant 
$\sum_{w\in S_n}(-1)^{\ell(w)}\x_q^w$ (with $x_{ij}=1$) and Dodgson's condensation to find the signed bigrassmannian statistic. 
We next wish to find the \emph{unsigned bigrassmannian statistic}; 
Reading \cite{reading} originally mentioned this problem. 
The natural idea is to consider the \eh{permanent} of the matrix $(q^{(i-j)^2/2})$. How can we evaluate this?
\item Recently, there are many references for research on bivariate permutation statistics such as \emph{Mahonian-Eulerian}; see Skandera \cite{sk2}, for example. Find the bivariate Mahonian-bigrassmannian statistics $\sum_{w\in S_n}t^{\ell(w)}q^{\b(w)}$. 


}\ee











\end{document}